\documentclass{amsart}[12]
\usepackage{mathrsfs, amsmath,amsfonts,amssymb}
\usepackage{amsmath}
\usepackage{mathtools}
\usepackage{graphicx}
\newtheorem{teo}{Theorem}
\newtheorem{pro}{Proposition}

\newtheorem{rem}{Remark}
\newtheorem{defi}{Definition}

\usepackage{fullpage}
\title{The Hyperbolic-type point process}

\author[N. Demni]{Nizar Demni}
\address{IRMAR, Universit\'e de Rennes 1\\ Campus de
Beaulieu\\ 35042 Rennes cedex\\ France}
\email{nizar.demni@univ-rennes1.fr}

\author[P. Lazag]{Pierre Lazag}
\address{I2M, CNRS, Aix-Marseille Universit\'e \\Technop\^ole de Ch\^ateau-Gombert \\ 13453 Marseille cedex 13 \\ France}
\email{pierre.lazag@univ-amu.fr}


\begin{document}
\maketitle
\begin{abstract}
In this paper, we introduce a two-parameters determinantal point process in the Poincar\'e disc and compute the asymptotics of the variance of its number of particles inside a disc centered at the origin and of radius $r$ as $r \rightarrow 1^-$. Our computations rely on simple geometrical arguments whose analogues in the Euclidean setting provide a shorter proof of Shirai's result for the Ginibre-type point process. In the special instance corresponding to the weighted Bergman kernel, we mimic the computations of Peres and Virag in order to describe the distribution of the number of particles inside the disc. 
\end{abstract}

\section{Introduction}
Determinantal point processes are random point measures on locally compact  polish spaces whose correlation functions are determinants of locally trace-class non negative operators bounded by one (\cite{Sos}). They appeared in Macchi's paper \cite{Mac} under the name `Fermion processes' since Fermions obey the Pauli exclusion principle so that their wavefunctions are given by Slater determinants. By the Macchi-Soshnikov Theorem (\cite{Sos}, see also \cite{ShiTak}), given a measure space $(E,\mu)$ and an orthogonal projection onto a closed subspace $F \subset L ^2(E,\mu)$ with Hermitian kernel $K$, there exists a determinantal point process whose correlation functions are governed by $K$ and whose number of particles is almost surely equal to the dimension of $F$. For instance, unitarily-invariant random matrix models give rise to determinantal point processes with almost surely finite numbers of particles (\cite{Soh}), while the Fock and the Bergman spaces provide examples of infinite-dimensional determinantal point processes. Actually, the former corresponds to the Ginibre process which is the weak limit of the eigenvalues process of the Ginibre matrix model (\cite{Gin}) and the latter corresponds to the zero set of the hyperbolic Gaussian analytic function whose matrix-valued extension is also the weak limit of the eigenvalues of square truncations of Haar unitary random matrices (\cite{Kri}). For other examples and various constructions of determinantal point processes, we refer the reader to \cite{Bor}.  

On the other hand, the Fock and the Bergman spaces may be realized as the null eigenspaces of the Schr\"odinger operators with a uniform magnetic field, known as the Landau Laplacian, in the complex plane and in the Poincar\'e disc respectively. In the flat geometrical setting, the Landau Laplacian has a discrete spectrum - Euclidean Landau levels - labeled by the set of non negative integers whose eigenspaces are infinite-dimensional and consist in general of polyanalytic functions. Besides, the corresponding reproducing kernels were computed in \cite{AIM} and subsequently used in \cite{Shi} in order to define the Ginibre-type point processes, where the author derives the asymptotics of the variance of the number of particles in a disc centered at the origin and of radius $r$ as $r \rightarrow \infty$. Note that by analogy with the Ginibre process, a finite-dimensional version of the Ginibre-type point process was introduced and studied in \cite{Hai-Hed} yet without any reference to an underlying random matrix model. As to the negatively-curved geometrical setting, the Landau Laplacian alwas admits a continuous spectrum and a discrete spectrum-hyperbolic Landau levels - arises as soon as the magnetic field strength is large enough. The corresponding eigenspaces are infinite-dimensional as well, and the expressions of their reproducing kernels are also available (see e.g. \cite{EGIM} and references therein).

In this paper, we use these kernels introduce the hyperbolic analogue of the Ginibre-type point process, and call it in a similar fashion `the hyperbolic-type' point process. Doing so allows to generalize the zero set of the hyperbolic Gaussian analytic function within the class of determinantal processes, in opposite to the Gaussian random series considered in \cite{Buc}. Furthermore, the hyperbolic-type determinantal point process converges weakly to the Ginibre-type point process if we let the curvature of the disc tends to zero which is in agreement with the geometrical contraction principle. Our main result establishes the exact asymptotics of the variance of the number of particles lying inside a disc centered at the origin and of radius $r$ as $r \rightarrow 1^-$. Though this is the hyperbolic analogue of Shirai's result for the Ginibre-type point process, our proof is completely different from Shirai's one and may even be adapted to the Ginibre-type point process in order to write a different proof of Shirai's result. More precisely, using the invariance of the reproducing kernels under appropriate groups of transformations - the translation group for the complex plane and the M\"obius group for the Poincar\'e disc - we are led to the computation of the Euclidean and the hyperbolic areas of some planar region. For the Ginibre-type point process, the area of this region is already known and yields directly Shirai's result. As to the hyperbolic-type point process, the computations are more involved than those in the Euclidean setting, nonetheless we succeed to express this area as a incomplete hypergeometric integral and to derive the sought asymptotics. Nonetheless, in the special instance corresponding to weighted Bergman kernels, we mimic the computations done in \cite{Per-Vir} and obtain the full description of the number of particles inside the disc.       

The paper is organized as follows. For sake of completeness, we recall in the next section the definition of the Ginibre-type point process and write another proof of Shirai's result using the invariance of the reproducing kernels under translations. In section 3, we introduce the hyperbolic-type point process and prove our main result. In the last section, we describe the distribution of the number of particles inside the disc in the case of weighted Bergman kernels which corresponds to the lowest hyperbolic Landau level.
 
\section{The Ginibre-type point process revisited}
Let $E$ be a locally compact polish space and let $\textrm{Conf}(E)$ be the space of locally finite configurations, that is, the space of all discrete subsets of $E$ having a finite number of elements in any compact set. We can equip $\textrm{Conf}(E)$ with the sigma-algebra $ \mathcal{F}(E)$ generated by the maps :
\begin{align*}
N_A :\quad \textrm{Conf}(E) &\rightarrow \mathbb{N} \\
X &\mapsto |X \cap A |
\end{align*}
for all relatively compact subsets $ A \subset E$. Then,
\begin{defi} 
A point process is a probability measure $\mathbb{P}$ on $(\textrm{Conf}(E),\mathcal{F}(E))$. It is a \textit{determinantal point process} with correlation kernel $K$ and reference measure $\mu$ on $(E, \mathcal{B}(E))$ if for every $n \in \mathbb{N}$ and every compactly-supported bounded function $f : E^n \rightarrow \mathbb{C}$, one has :
\begin{align*}
\mathbb{E}_{\mathbb{P}} \left[ \sum_{x_1, \dots x_n \in X} f(x_1,\dots,x_n) \right] = \int_{E^n}f(x_1,...,x_n) \det \left(K(x_i,x_j) \right)_{1 \leq i , j \leq n} d\mu(x_1)...d\mu(x_n).
\end{align*}
Here, the sum in the left-hand side is over all simple $n$-points in the random configuration $X$.
\end{defi}

A well-known example is the Ginibre point process corresponding to the following data: 
\begin{itemize}
\item $E = \mathbb{C}$. 
\item $d\mu(z) = e^{-|z|^2}dz/\pi,$ $dz$ being the Lebesgue measure in $\mathbb{C}$. 
\item $F$ is the Fock space consisting of entire functions in the Hilbert space
\begin{equation*}
L^2(\mathbb{C}, e^{-|z|^2}dz/\pi)
\end{equation*}
whose reproducing kernel is $K_0(z,w) = e^{z\overline{w}}$. 
\end{itemize}
It arises as the weak limit of the eigenvalues process of the Ginibre random matrix model as the size of the matrix tends to infinity. On the other hand, the Fock space may be realized as the null space of the Euclidean Landau Laplacian with uniform magnetic field\footnote{Without loss of generality, the magnetic field strentgh may be taken equal to one.}: 
\begin{equation} \label{ELL}
-\frac{\partial^2}{\partial z \partial \overline{z}} + \overline{z}\frac{\partial}{\partial \overline{z}}
\end{equation}
which has discrete spectrum given by non negative integers (\cite{AIM}). Besides, for any $n \in \mathbb{N}$, the $n$-th eigenspace consists of polyanalytic functions which are solutions of the generalized Cauchy-Riemann equation: 
\begin{equation*}
\left(\frac{\partial}{\partial \overline{z}}\right)^{n+1} f = 0, \quad n \geq 0,
\end{equation*}
and its reproducing kernel reads (\cite{AIM}): 
\begin{equation*}
K_n(z,w) = e^{z\overline{w}} L_n(|z-w|^2), \quad z,w \in \mathbb{C},
\end{equation*}
where $L_n$ is the $n$-th Laguerre polynomial. In \cite{Shi}, the author introduced the Ginibre-type point process $\mathbb{P}_n$ at level $n$ as the determinantal point process with kernel $K_n$. There, he also computed the variance of the number of its particles inside a disc $D_r$ centered at the origin and of radius $r$, and showed that it grows linearly 
as $r \rightarrow \infty$. In this respect, recall that the reproducing property leads to the following formula: 
\begin{pro}
Let $\mathbb{P}$ be a determinantal point process defined through a reproducing kernel $K$ with respect to a reference measure $\mu$. Then, for any relatively compact set $A \subset E$, the variance of $N_A$ is giveny by:
\begin{align} \label{VF}
\textrm{Var}_{\mathbb{P}}(N_A)=\int_A d\mu(z) \int_{E \setminus A} d\mu(w) |K(z,w)|^2
\end{align}
\end{pro}

With the help of \eqref{VF} and the convolution property of Laguerre polynomials, the following result was proved in \cite{Shi}:
\begin{equation*}
\textrm{V}_n(N_r) = \frac{r}{\pi}\int_0^{\infty} dt |L_n(t)|^2 e^{-t} \int_0^{t \wedge 4r^2} \sqrt{1 - \frac{x}{4r^2}} \frac{dx}{\sqrt{x}}.
\end{equation*}
Here, $\textrm{V}_n$ stands for the variance with respect to $\mathbb{P}_n$, and $N_r$ denotes - here and after - the number of particles inside a disc $D_r$ centered at the origin and of radius $r >0$. In the sequel, we write a shorter proof of this result which has the merit to apply to the hyperbolic-type point process introduced later since it relies on geometrical arguments. To this end, perform the variables change $t \rightarrow t^2, x \rightarrow (2rx)^2$ in order to rewrite Shirai's formula as 
\begin{align}\label{ShiFor}
\textrm{V}_n(N_r) & = 8 \pi r^2 \int_0^{\infty}t dt |L_n(t^2)|^2 \frac{e^{-t^2}}{\pi^2} \int_{0}^{(t/2r) \wedge 1} \sqrt{1 - x^2}dx \nonumber
\\& =  \int_{\mathbb{C}} |L_n(|z|^2)|^2 \frac{e^{-|z|^2}}{\pi^2} \int_{0}^{(|z|/2r) \wedge1}(4 r^2) \sqrt{1 - x^2}dx.
\end{align} 
On the hand, start from \eqref{VF} and use the invariance under translations of $K_n(z,w)\mu(dw)\mu(dz)$ to get: 
\begin{align*}
\textrm{V}_n(N_r) & = \int_{D_r^c}\mu(dw) \int_{D_r} \mu(dz) |K_n(z,w)|^2 
\\& = \frac{1}{\pi^2} \int_{D_r^c}dw  \int_{D_r} dz e^{-|z-w|^2} |L_n(|z-w|^2)|^2
 \\& = \frac{1}{\pi^2} \int_{D_r^c}dw  \int_{w+ D_r} dz e^{-|z|^2} |L_n(|z|^2)|^2
 \\& = \frac{1}{\pi^2} \int_{\mathbb{C}}dz  e^{-|z|^2} |L_n(|z|^2)|^2  \int_{D_r^c \cap \{w, |w-z| < r\}} dw
 \\& = \frac{1}{\pi^2} \int_{\mathbb{C}}dz  e^{-|z|^2} |L_n(|z|^2)|^2  \textrm{Area}(D_r^c \cap D_r(z)), 
 \end{align*}
where $D_r(z)$ is the disc centered at $z$ and of radius $r$. If $|z| \geq 2r$, then $\textrm{Area}(D_r^c \cap D_r(z)) = \pi r^2$ since $D_r(z) \subset D_r^c$, which coincides with the value of the inner integral in \eqref{ShiFor}: 
 \begin{equation*}
 4r^2  \int_{0}^{(|z|/2r) \wedge1} \sqrt{1 - x^2}dx = 4r^2  \int_{0}^{1} \sqrt{1 - x^2}dx = \pi r^2. 
 \end{equation*}
 Otherwise, if $|z| < 2r$, then $D_r^c \cap D_r(z)$ is the complementary in $D_r(z)$ of the overlapping of the discs $D_r$ and $D_r(z)$ and its area  is known to be equal to
\begin{align*}
\textrm{Area}(D_r^c \cap D_r(z)) &= \pi r^2 - 2r^2 \arccos\left(\frac{|z|}{2r}\right) + \frac{|z|}{2} \sqrt{4r^2 - |z|^2} 
\end{align*}
Again, this area coincides as well with the value of the inner integral in \eqref{ShiFor}: 
\begin{equation*}
4 r^2 \int_{0}^{(|z|/2r)} \sqrt{1 - x^2}dx = \pi r^2 -  2r^2\int_0^{\arccos(|z|/2r)} (1-\cos(2\theta))d\theta. 
\end{equation*}
Shirai's formula is proved. In a nutshell, the computations of $\textrm{V}_n(N_r)$ relies essentially on the invariance under translation of the integrand, on the transitive action of the translation group on $\mathbb{C}$ and on the knowledge of the expression of the Euclidean area of $D_r^c \cap D_r(z)$. As we shall see in the next section, the situation is very similar in the hyperbolic setting, yet the computations are tricky and involved.  

\section{The hyperbolic-type point process}
In this section, we introduce the hyperbolic-type point process. To this end, we recall from \cite{EGIM} the spectral decomposition of the hyperbolic Landau laplacian with uniform magnetic field\footnote{Unlike the Euclidean setting, the strength $\nu$ of the magnetic field can not be reduced to one.} $\nu \geq 0$. Let $\mathbb{D}$ be the unit disc, then the hyperbolic Landau laplacian is the following differential operator acting on smooth functions as: 
\begin{align*}
H_\nu := -4(1-z\bar{z})\left( (1-z\bar{z})\frac{\partial^2}{\partial z \partial \bar{z}} - 2\nu \bar{z}\frac{\partial}{\partial \bar{z}}\right).
\end{align*}
It is a densely defined operator in  $L^2(\mathbb{D},\lambda_{\nu})$ where 
\begin{equation*}
\lambda_{\nu}(dz) := (1-|z|^2)^{2\nu-2} dz,
\end{equation*}
and admits a unique self-adjoint extension which we also denote by the same symbol $H_\nu$. Besides, its spectrum has a purely continuous part $[1,+ \infty [$ and if $\nu > 1/2$, then a non negative discrete part arises and consists  of the so-called hyperbolic Landau levels:
\begin{align*}
\epsilon^{\nu}_m=4m(2\nu-m-1) ; \quad \quad m=0,1,...,[\nu -1/2], \quad 2(\nu-m) - 1 \neq 0,
\end{align*}
$[x]$ being the largest integer less than or equal to $x$. The corresponding eigenspaces are infinite-dimensional and the reproducing kernel $G^{\nu}_m$ associated with a given hyperbolic Landau level $\epsilon^{\nu}_m$ reads (\cite{EGIM}):
\begin{equation}
G_m^{\nu}(z,w) = \frac{2(\nu-m)-1}{\pi}(1-z\bar{w})^{-2\nu}\left( \frac{|1-z\bar{w}|^2}{(1-|z|^2)(1-|w|^2)}\right)^m P_m^{(0,2(\nu-m)-1)}\left(2\frac{(1-|z|^2)(1-|w|^2)}{|1-z\bar{w}|^2}-1\right)
\end{equation}
where $P_m^{(0,2(\nu-m)-1)}$ is the $m$-th Jacobi polynomial (\cite{AAR}). With these data in hands, we are ready to introduce the hyperbolic-type point process :
\begin{defi}
Let $\nu > 1/2$ and $m \in \{0,...,[\nu -1/2] \}$. The hyperbolic-type point process $\mathbb{P}^{\nu}_m$ at level $\epsilon_m^{\nu}$ is the determinantal point process with correlation kernel $G^{\nu}_m$.  
\end{defi}
In the sequel, we will denote the variance with respect to $\mathbb{P}^\nu_m$ by $\textrm{V}^\nu_m$. Notice that if $2\nu \geq 2$ is an integer and $m=0$, then $\mathbb{P}^{\nu}_0$ reduces to the singular locus of a complex Gaussian 
$(2\nu -1) \times (2\nu-1)$ matrix-valued random series in which case $G_0^{\nu}$ is the weighted Bergman kernel in $\mathbb{D}$ (\cite{Kri}). In particular, $\mathbb{P}^{1}_0$ is nothing else but the hyperbolic Gaussian determinantal process defined and studied in \cite{Per-Vir}, and is realized as the zeros of a Gaussian analytic series. In particular, it was proved there that 
\begin{equation}\label{PV}
\textrm{V}^1_0(N_r) = \frac{r^2}{1-r^4},
\end{equation}
and another proof of this result is given in \cite{Buc}. However, beware that the Gaussian analytic series in the disc studied in \cite{Buc} are in general not determinantal and that asymptotic formulas for the variance of $N_r$ are derived there from a formula quite similar to \eqref{VF}. More generally, the adaptation of our previous proof of Shirai's result yields the following formula for $\textrm{V}_m^\nu(N_r)$:  

\begin{pro}\label{pro2}
Let $\mathbb{P}^{\nu}_m$ be the hyperbolic-type point process at level $\epsilon^{\nu}_m$. Set
\begin{equation*}
f_{\nu,m}(x) :=  \left[\frac{(2(\nu-m)-1)}{\pi}(1-\tanh^2(x))^{\nu-m}P_m^{(0,2(\nu-m)-1)}\left(1-2\tanh^2(x)\right)\right]^2,
\end{equation*}
and recall the hyperbolic distance: 
\begin{equation*}
\cosh^2(d(z,w)) := \frac{|1-z\overline{w}|^2}{(1-|z|^2)(1-|w|^2)} = \frac{1}{1- \tanh^2(d(z,w))}.
\end{equation*}
Then,
\begin{equation}\label{Int1}
\textrm{V}_m^\nu(N_r)= 2\int_\mathbb{D} \lambda_0(dz) f_{\nu,n}(d(z,0))\int_{r \vee (||z|-r|)/(1-|z|r)}^{(|z| + r)/(1+|z| r)} \arccos \left( \frac{t^2+|C_{z,r}|^2-R_{z,r}^2}{2t|C_{z,r}|} \right)\frac{tdt}{(1-t^2)^2},
\end{equation}
where 
\begin{align*}
C_{z,r} := \frac{1-r^2}{1-|z|^2r^2}z \quad , \quad R_{z,r} := \frac{1-|z|^2}{1-|z|^2r^2}r.
\end{align*}
\end{pro}

Before proving this proposition, we state the main result of our paper:

\begin{teo} \label{teo1}
The following asymptotics of $\textrm{V}_m^\nu(N_r)$ holds: 
\begin{equation*}
\textrm{V}_m^\nu(N_r) \sim \frac{C^{\nu}_m}{1-r^2}, \qquad r \rightarrow 1^-,
\end{equation*}
where 
\begin{align*}
C^{\nu}_m & =  \int_{\mathbb{D}}\lambda_0(dz)f_{\nu,m}(d(z,0))\arccos(1-2|z|^2).
\end{align*}
\end{teo}

\begin{proof}[Proof of Proposition \ref{pro2}]
Firstly, it is straightforward that: 
\begin{equation} \label{eqvarhyp}
|G_m^{\nu}(z,w)|^2\lambda_{\nu}(dz)\lambda_{\nu}(dw) =  f_{\nu, m}(d(z,w)) \lambda_{0}(dz)\lambda_{0}(dw).
\end{equation}
Secondly, the hyperbolic area measure $\lambda_0$ and the hyperbolic distance are invariant under the action of the M\"obius transformations: 
\begin{equation*}
g_{w,\theta}:  z \mapsto e^{i\theta} \frac{w-z}{1-\overline{w}z} , \quad w \in \mathbb{D}, \theta \in [0,2\pi],
\end{equation*}
which act transitively on $\mathbb{D}$, see e.g. \cite{Zhu}. Note that $g_{w,0}$ is an involution, maps the origin to $w$ and may be written as 
\begin{equation*}
g_{w,0} = R_{\arg(w)}g_{|w|, 0}R_{-\arg(w)},
\end{equation*}
where 
\begin{equation*}
R_{\theta} = \left(\begin{array}{lr}
e^{i\arg(w)/2} & 0 \\ 
0 & e^{-i\arg(w)/2}
\end{array}\right).
\end{equation*}

Denoting simply $g_{w} = g_{w,0}$, it follows from \eqref{VF} that: 
\begin{align*}
\textrm{V}_m^\nu(N_r) & = \int_{D_r^c} \lambda_{\nu}(dw) \int_{D_r} \lambda_{\nu}(dz) |G_m^{\nu}(z,w)|^2 
\\& = \int_{D_r^c} \lambda_{0}(dw) \int_{D_r} \lambda_{0}(dz) f_{\nu, m}(d(z,g_w 0))
\\& =  \int_{D_r^c} \lambda_{0}(dw) \int_{g_wD_r} \lambda_{0}(dz) f_{\nu, m}(d(z,0)).
\\& =  \int_{D_r^c} \lambda_{0}(dw) \int_{R_{\arg(w)}g_{|w|} D_r} \lambda_{0}(dz) f_{\nu, m}(d(z,0)).
\end{align*}
Now, write
\begin{equation*}
g_{|w|} z = \frac{1}{|w|} \left[1- \frac{1-|w|^2}{|w|z+1}\right], 
\end{equation*}
then the image of $D_r$ under the inversion $z \mapsto 1/(|w|z+1)$ is the disc 
\begin{equation*}
D\left(\frac{1}{1-|w|^2r^2}, \frac{|w| r}{1-|w|^2r^2}\right), 
\end{equation*}
and in turn,
\begin{equation*}
g_{|w|}D_r = D\left(|C_{w,r}|, R_{w,r}\right). 
\end{equation*}
Note that when $|w| \rightarrow 1$ then $g_{|w|}D_r$ tends to the point $\{1\}$ while
\begin{equation*}
g_{r}D_r = D\left(\frac{r}{1+r^2}, \frac{r}{1+r^2}\right). 
\end{equation*}
Consequently, 
\begin{equation*}
\bigcup_{w \in D_r^c} g_w D_r = \bigcup_{w \in D_r^c}  R_{\arg(w)}g_{|w|}R_{-\arg(w)} D_r = \mathbb{D}. 
\end{equation*}
Moreover, given $z \in \mathbb{D}$, then $g_w(y) = z$ for some $y \in D_r$ is equivalent to $g_w(z) = y$ since $g_w$ is an involution. But, 
\begin{align*}
\{ w \in \mathbb{D} \quad ; \quad |g_w(z)| < r \}= D(C_{z,r}, R_{z,r} )
\end{align*}
is the disc centered at $C_{z,r}$ and of radius $R_{z,r}$. As a result, Fubini Theorem entails:
\begin{align*}
\textrm{V}_m^\nu(N_r)=\int_\mathbb{D} \lambda_0(dz) f_{\nu,m}(d(z,0)) \int_{D_r^c \cap D(C_{z,r},R_{z,r})} \lambda_0(dw).
\end{align*}
Finally, consider the inner integral and note that it does not depend on $\arg(z)$. Using polar coordinates, the range of $\arg(w), w \in D_r^c \cap D(C_{z,r},R_{z,r})$ may be determined (for fixed $|w|$) from Al Kashi's theorem applied to the triangle formed by the origin, $C_ {z,r}$ and one of the intersection points of $\partial D(0,t)$ and $\partial D(C_{z,r},R_{z,r})$, see Fig.1 below.
\begin{figure}[hbtp]
\centering
\includegraphics[scale=0.3]{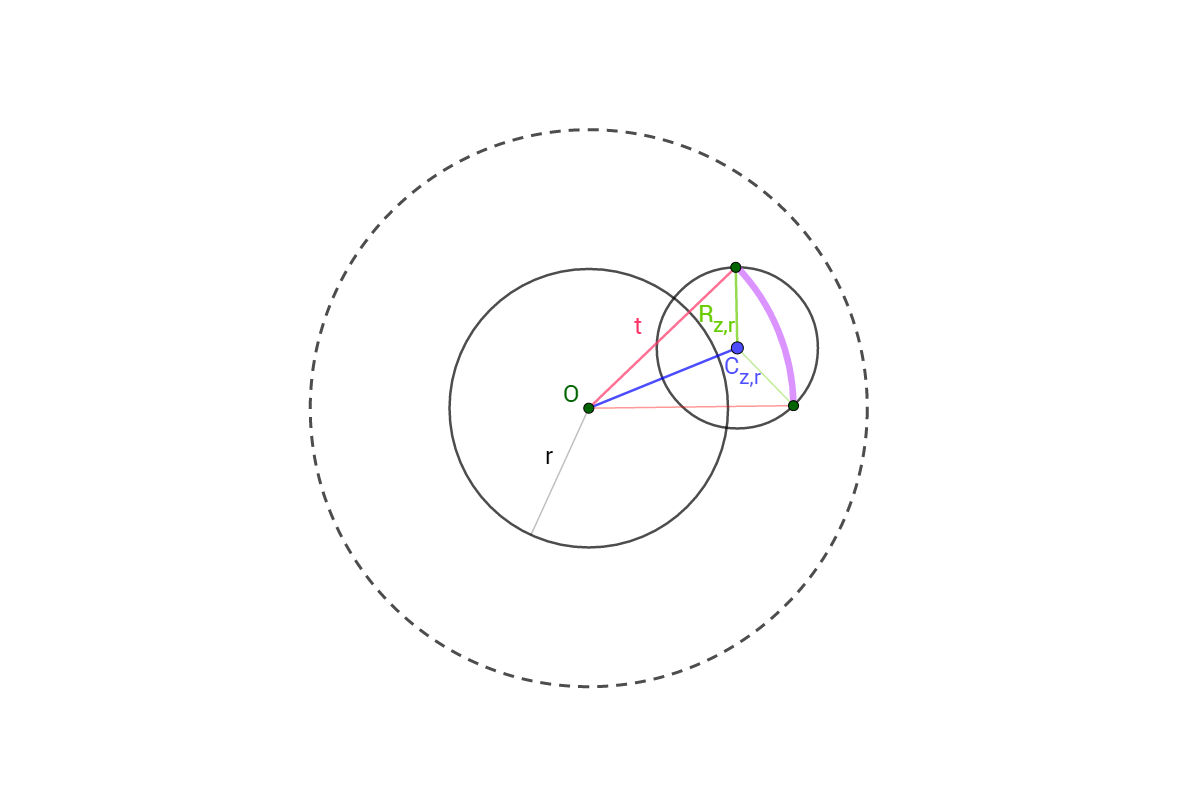}
\caption{range of $\arg(w)$ for fixed $|w| = t$ (in purple)}
\end{figure}

As to $|w|$, its range is determined as follows. The closest point to the origin and lying in $\partial D(C_{z,r},R_{z,r})$ is:
\begin{align*}
C_{z,r}-R_{z,r}\frac{C_{z,r}}{|C_{z,r}|}
\end{align*}
and its modulus is given by:
\begin{align*}
\left| |C_{z,r}|- R_{z,r} \right| = \frac{||z|-r|}{1-|z|r}.
\end{align*}
Similarly, the most distant point is:
\begin{align*}
C_{z,r}+R_{z,r}\frac{C_{z,r}}{|C_{z,r}|}
\end{align*}
and its modulus is given by :
\begin{align*}
|C_{z,r}|+ R_{z,r} = \frac{|z| + r}{1+|z| r}.
\end{align*}
Hence
\begin{align*}
w \in D(C_{z,r},R_{z,r}) \Rightarrow \left| |C_{z,r}|- R_{z,r} \right| < |w| < |C_{z,r}|+ R_{z,r},
\end{align*}
and consequently :
\begin{align*}
w \in D_r^c \cap D(C_{z,r},R_{z,r}) \Rightarrow  \left| |C_{z,r}|- R_{z,r} \right| \vee r < |w| < |C_{z,r}|+ R_{z,r}.
\end{align*}
Altogether yields :
\begin{align*}
\int_{\left| |C_{z,r}|-R_{z,r} \right| \vee r}^{|C_{z,r}|+R_{z,r}} 2\arccos \left( \frac{t^2+|C_{z,r}|^2-R_{z,r}^2}{2t|C_{z,r}|} \right)\frac{tdt}{(1-t^2)^2}.
\end{align*}
Since 
\begin{equation*}
|R_{z,r} - |C_{z,r}|| = \frac{||z|-r|}{1-|z|r}, \quad  |C_{z,r}|+R_{z,r} = \frac{|z| + r}{1+|z| r},
\end{equation*}
the proposition is proved. 
\end{proof}

\begin{rem}[Contraction principle]
From the very definition of $f_{\nu, m}$, we derive
\begin{equation*}
f_{\nu, m}(d(z,0)) =  \left[\frac{2(\nu-m)-1}{\pi}(1-|z|^2)^{\nu-m} P_m^{(0,2(\nu-m)-1)}\left(1-2|z|^2\right)\right]^2.
\end{equation*}
Let $R > 1$ be a positive real number and perform in \eqref{Int1} the variable change $z \mapsto z/R$: 
\begin{multline*}
\textrm{V}_m^\nu(N_r) = \left(\frac{2(\nu-m)-1}{\pi R^2}\right)^2 \int_{\mathbb{D}_R} \left(1-\frac{|z|^2}{R^2}\right)^{2\nu-2m-2}\left[P_m^{(0,2(\nu-m)-1)}\left(1-2\frac{|z|^2}{R^2}\right)\right]^2 
\\ \int_{D_r^c \cap D(C_{z/R,r},R_{z/R,r})} \lambda_0(dw),
\end{multline*}
where $\mathbb{D}_R$ is the disc centered at the origin and of radius $R$. Now, rescale $r \mapsto r/R$, take $\nu = R^2/2$ and perform the variable change $w \mapsto w/R$ in the inner integral. Then: 
\begin{multline*}
\textrm{V}_m^{R^2/2}(N_{r/R}) = \left(\frac{(R^2-2m)-1}{\pi R^2}\right)^2 \int_{\mathbb{D}_R} \left(1-\frac{|z|^2}{R^2}\right)^{R^2-2m-2}\left[P_m^{(0,R^2-2m-1)}\left(1-2\frac{|z|^2}{R^2}\right)\right]^2 
\\ \int_{D_r^c \cap D(\tilde{C}_{z,r},\tilde{R}_{z,r})} \frac{\lambda_0(dw)}{R^2},
\end{multline*}
where 
\begin{equation*}
\tilde{C}_{z,r} = \frac{1-(r/R)^2}{1-(|z|r/R^2)^2}z, \quad \tilde{R}_{z,r} = \frac{1-(|z|/R)^2}{1-(|z|r/R^2)^2}r. 
\end{equation*}
Using the limiting relation (\cite{AAR}): 
\begin{equation*}
\lim_{R \rightarrow \infty} P_m^{(0,R^2-2m-1)}\left(1-2\frac{|z|^2}{R^2}\right) = L_m(|z|^2), 
\end{equation*}
it follows that 
\begin{equation*}
\lim_{R \rightarrow \infty} R^2 \textrm{V}_m^{R^2/2}(N_{r/R}) = \frac{1}{\pi^2} \int_{\mathbb{C}} e^{-|z|^2} \left[L_m(|z|^2)\right]^2  \int_{D_r^c \cap D_r(z)} dw = \textrm{V}_n(N_r).
\end{equation*}
Such a result is expected to hold by the virtue of the geometrical contraction principle. However, what is not expected and less obvious is the rescaling of the magnetic field strength. 
\end{rem}
 
Now, we shall prove Theorem \ref{teo1}.  
\begin{proof}
The integral in the RHS of \eqref{Int1} may be transformed after performing an integration by parts into: 
\begin{multline}\label{F1}
 \int_{\left||C_{z,r}|-R_{z,r} \right| \vee r}^{|C_{z,r}|+R_{z,r}} \frac{1}{\sqrt{4|C_{z,r}|^2t^2 - (t^2+|C_{z,r}|^2-R_{z,r}^2)^2}} \frac{(t^2 + R_{z,r}^2- |C_{z,r}|^2)}{t(1-t^2)} dt \\ 
- \frac{1}{(1-r^2)}\arccos\left( \frac{r^2+|C_{z,r}|^2-R_{z,r}^2}{2r|C_{z,r}|} \right) {\bf 1}_{\{|z| < 2r/(1+r^2)\}}.
\end{multline}
Performing further the variables change $t \mapsto t^2$, the integral above becomes 
\begin{multline*}
\frac{1}{2} \int_{E_{z,r} \vee r^2}^{F_{z,r}} \frac{1}{\sqrt{-t^2 +2A_{z,r}t - B_{z,r}^2}} \frac{(t + B_{z,r})}{t(1-t)}dt 
\end{multline*}
where 
\begin{equation*}
A_{z,r} : = |C_{z,r}|^2+R_{z,r}^2, \quad B_{z,r} := R_{z,r}^2 - |C_{z,r}|^2, 
\end{equation*} 
\begin{equation*}
E_{z,r} := (R_{z,r} - |C_{z,r}|)^2, \quad F_{z,r} := (R_{z,r} + |C_{z,r}|)^2.
\end{equation*}
The trinomial inside the square root factorizes as 
\begin{equation*}
-t^2 +2A_{z,r}t - B_{z,r}^2 = (t-E_{z,r})(F_{z,r} - t),
\end{equation*}
and the following decomposition holds 
\begin{equation*}
\frac{(t + B_{z,r})}{t(1-t)} = \frac{B_{z,r}}{t} + \frac{1+B_{z,r}}{1-t}.
\end{equation*}
We are then led to: 
\begin{align*}
\frac{1}{2} \int_{E_{z,r} \vee r^2}^{F_{z,r}} \frac{1}{\sqrt{(t-E_{z,r})(F_{z,r} - t)}}\left[\frac{B_{z,r}}{t} + \frac{1+B_{z,r}}{1-t}\right] dt, 
\end{align*}
which may be turned into a incomplete hypergeometric-type integral after the variable change 
\begin{equation*}
t \mapsto \frac{F_{z,r} - t}{F_{z,r}-E_{z,r}}.
\end{equation*}
Indeed, quick computations yield:
\begin{align}\label{Int2}
\frac{1}{2} \int_{0}^{1 \wedge H_{z,r}} \frac{dt}{\sqrt{t(1-t)}}\left[\frac{B_{z,r}}{F_{z,r} - (F_{z,r} - E_{z,r})t} + \frac{1+B_{z,r}}{(F_{z,r} - E_{z,r})t + 1-F_{z,r}}\right] 
\end{align}
where 
\begin{equation*}
H_{z,r} := \frac{F_{z,r} - r^2}{F_{z,r} - E_{z,r}} = \frac{(2r+|z|(1+r^2))(1-|z|r)^2}{4r(1-|z|^2)} .
\end{equation*}
We similarly compute 
\begin{equation*}
F_{z,r} - E_{z,r} = 4 \frac{|z|(1-|z|^2)r(1-r^2)}{(1-|z|^2r^2)^2}, \quad 1-F_{z,r} = \frac{(1-|z|^2)(1-r^2)}{(1+|z|r)^2},
\end{equation*}
and set 
\begin{equation*}
U_{z,r} := \frac{F_{z,r} - E_{z,r}}{1-F_{z,r}} = \frac{4|z|r}{(1-|z|r)^2}, \quad V_{z,r} := \frac{F_{z,r}-E_{z,r}}{F_{z,r}} = 4\frac{|z|(1-|z|^2)r(1-r^2)}{(|z|+r)^2(1-|z|r)^2}.
\end{equation*}
As a result, \eqref{Int2} may be written as 
\begin{align}\label{Int3}
\frac{1}{2} \int_{0}^{1 \wedge H_{z,r}} \frac{dt}{\sqrt{t(1-t)}}\left[\frac{B_{z,r}}{F_{z,r}} \frac{1}{1 - V_{z,r}t} + \frac{1+B_{z,r}}{1-F_{z,r}}\frac{1}{1+ U_{z,r} t}\right].
\end{align}
Now, we are ready to let $r \rightarrow 1^-$ in \eqref{Int3}. More precisely, we obviously have
\begin{equation*}
\lim_{r \rightarrow 1^-} C_{z,r} = 0, \quad \lim_{r \rightarrow 1^-} R_{z,r} = 1, \quad \lim_{r \rightarrow 1^-} H_{z,r} = \frac{1-|z|}{2},
\end{equation*}
so that
\begin{equation}\label{F2}
\lim_{r \rightarrow 1^-} \frac{B_{z,r}}{F_{z,r}} \int_{0}^{1 \wedge H_{z,r}} \frac{dt}{\sqrt{t(1-t)}}\frac{1}{1 - V_{z,r}t} = \int_0^{(1-|z|)/2}\frac{dt}{\sqrt{t(1-t)}} = 2\arcsin\sqrt{\frac{1-|z|}{2}} = \arccos(|z|).
\end{equation}
Similarly 
\begin{align*}
\lim_{r \rightarrow 1^-} (1-r^2) \frac{1+B_{z,r}}{2(1-F_{z,r})}\int_{0}^{1 \wedge H_{z,r}} \frac{dt}{\sqrt{t(1-t)}}\frac{1}{1 +U_{z,r}t} & = \frac{(1+|z|)^2}{(1-|z|^2)} \int_0^{(1-|z|)/2}\frac{dt}{\sqrt{t(1-t)}}\frac{1}{1 +U_{z,1}t}
\\& =  \frac{2(1+|z|)}{(1-|z|)} \int_0^{(1/2) \arccos(|z|)} \frac{1}{1 +U_{z,1}\sin^2 t} dt
\\& =  \int_0^{\arccos(|z|)} \frac{(1-|z|^2)}{(1-|z|)^2 + 4|z|\sin^2 (t/2)} dt
\\& = \int_0^{\arccos(|z|)} \frac{(1-|z|^2)}{1+|z|^2 -2|z|\cos t} dt
\\& = 2\arctan \left(\frac{1+|z|}{1-|z|}\tan\left[\frac{\arccos(|z|)}{2}\right] \right),
\end{align*}
where the last equality follows from formula 2.556 (1) in \cite{Gra-Ryz}. Since
\begin{equation*}
\tan(u/2) = \frac{\sin(u)}{1+\cos(u)}
\end{equation*}
then 
\begin{align}\label{F3}
2\arctan \left(\frac{1+|z|}{1-|z|}\tan\left[\frac{\arccos(|z|)}{2}\right] \right) & = 2\arctan\left(\sqrt{\frac{1+|z|}{1-|z|}}\right) \nonumber 
\\& = \pi - 2\arccos\left(\sqrt{\frac{1+|z|}{2}}\right) \nonumber
\\& = \pi - \arccos(|z|).
\end{align}
Finally,
\begin{equation*}
\frac{r^2+|C_{z,r}|^2-R_{z,r}^2}{2r|C_{z,r}|} = \frac{|z|(1+r^2)}{2r} 
\end{equation*}
whence we deduce 
\begin{align}\label{F4}
\lim_{r \rightarrow 1^-}\arccos\left( \frac{r^2+|C_{z,r}|^2-R_{z,r}^2}{2r|C_{z,r}|} \right) {\bf 1}_{\{|z| < 2r/(1+r^2)\}} & = \arccos(|z|){\bf 1}_{\{|z| < 1\}}. 
\end{align}
Gathering \eqref{F1}, \eqref{F2}, \eqref{F3} and \eqref{F4}, we deduce that:
\begin{align*}
C^{\nu}_m & =  \int_{\mathbb{D}}\lambda_0(dz)f_{\nu,m}(d(z,0))\left[\pi - 2 \arccos(|z|)\right],
\\& = \int_{\mathbb{D}}\lambda_0(dz)f_{\nu,m}(d(z,0))\arccos(1-2|z|^2),
\end{align*}
which proves Theorem \eqref{teo1}. 
\end{proof}

\begin{rem}
For $m \in \{0, \dots, [\nu -(1/2)]$, let 
\begin{equation*}
\mathcal{A}_m^{\nu}(\mathbb{D}) = \{h \in L^2(\mathbb{D}, \lambda_{\nu}), H_{\nu}h  = \epsilon_m^{\nu} h\}
\end{equation*}
be the corresponding generalized Bergman space. An orthogonal basis $(\phi_m^{\nu}(j))_{j \geq 0}$ of $\mathcal{A}_m^{\nu}(\mathbb{D})$ was given for instance in \cite{EGIM}, eq. (2,2), and one readily sees that 
\begin{equation*}
C^{\nu}_m = \left(\frac{2(\nu-m)-1)}{\pi}\right)^2 \int_{\mathbb{D}}\lambda_{\nu}(dz)[\phi_m^{\nu}(m)]^2\arccos(1-2|z|^2).
\end{equation*}
By the virtue of eq. (2.4) in \cite{EGIM}, we further get the following bound: 
\begin{equation*}
C^{\nu}_m \leq \frac{(2(\nu-m)-1)^2}{\pi} ||\phi_m^{\nu}(m)||^2_{L^2(\mathbb{D}, \lambda_{\nu})} = 2(\nu-m) - 1.
\end{equation*}
\end{rem}

\section{The weighted Bergman kernel and the lowest hyperbolic Landau level}
Specializing Theorem \ref{teo1} with $m=0$ and integer values of $2\nu$, the asymptotics as $r \rightarrow 1^-$ of the variance of the number of zeros in $D_r$ of the complex Gaussian matrix-valued series. Nonetheless, we may mimick the proofs of Lemma 14 when the discs coincide and of Theorem 2 (i) in \cite{Per-Vir} in order to derive the distribution of $N_r$ for any real $\nu > 1/2$. More precisely,
\begin{teo}
For any $s \in (-1,1)$, 
\begin{equation*}
\mathbb{E}((1+s)^{N_r}) = \prod_{j=1}^{\infty} \left(1+s(2\nu-1)\frac{(2\nu)_{j-1}}{(j-1)!}B_{r} (j, 2\nu-1)\right),
\end{equation*}
where 
\begin{equation*}
B_r(j, 2\nu-1) := \int_0^{r^2} s^{j-1}(1-s)^{2\nu-2} ds,
\end{equation*}
is the incompete Beta integral. In particular, $N_r$ has the same distribution as 
\begin{equation*}
\sum_{j=1}^{\infty}X_j
\end{equation*}
where $(X_j)_{j \geq 1}$ are independent $\{0,1\}$-valued random variables defined on some probability space $(\Omega, \mathbb{F}, \mathbb{Q})$ such that 
\begin{equation*}
\mathbb{Q}(X_j = 1) = (2\nu-1)\frac{(2\nu)_{j-1}}{(j-1)!}B_{r} (j, 2\nu-1) = \frac{B_{r} (j, 2\nu-1)}{B_{1} (j, 2\nu-1)}.
\end{equation*}
\end{teo}
\begin{proof}
Since the proof follows the lines written p. 15-16 in \cite{Per-Vir}, we shall only indicate what modifications should be performed there. Firstly, the statement of Lemma 14 for identical discs is modified as follows: the weighted Bergman kernel admits the following expansion: 
\begin{equation*}
G_0^{\nu}(z,w) = \frac{2\nu-1}{\pi}\frac{1}{(1-z\bar{w})^{2\nu}} = \frac{2\nu-1}{\pi} \sum_{j \geq 0} \frac{(2\nu)_j}{j!} (z\overline{w})^j.
\end{equation*}
But since $\lambda_{\nu}$ is radial, then the discussion in the bottom of p.15 is still valid and yields the binomial moments of $N_r$: for any $k \geq 1$, 
\begin{align}\label{BinMom}
\mathbb{E}_0^{\nu} (N_r(N_r-1)(N_r - k+1)) & = \int_{D_r^k} \det(G_0^{\nu}(z_l,z_j))_{1 \leq m,l \leq k} \,\, dz_1 \dots dz_k \nonumber 
\\& = \sum_{\sigma \in S_k} \prod_{\tau \in \sigma} (-1)^{|\tau|+1}\sum_{j \geq 0} \left[\frac{(2\nu-1)(2\nu)_j}{j!}B_r(j+1, 2\nu-1) \right]^{|\tau|}
\end{align}
where $S_k$ is the symmetric group of $\{1, \dots, k\}$, $|\tau|$ is the size of the cycle $\tau$ in the permutation $\sigma$.  Secondly, we come to the proof of Theorem (i): set 
\begin{equation*}
\beta_k^{\nu}:= \mathbb{E}\left(\binom{N_r}{k}\right) = \frac{1}{k!}\sum_{\sigma \in S_k} \prod_{\tau \in \sigma} (-1)^{|\tau|+1}\sum_{j \geq 0} \left[\frac{(2\nu-1)(2\nu)_j}{j!}B_r(j+1, 2\nu-1) \right]^{|\tau|},
\end{equation*}
and 
\begin{equation*}
\beta^{\nu}(s) = \sum_{k \geq 0} \beta_k^{\nu}s^k, \quad s \in (-1,1), \quad \beta_0^{\nu} = 1. 
\end{equation*}
Then, the recurrence equation (32) still holds and may be derived by simply separating the block of a given permutation $\sigma$ containing $1$ from the others: 
\begin{equation*}
\beta_k^{\nu} = \frac{1}{k} \sum_{l=1}^k (-1)^{l+1} \sum_{j \geq 0} \left[\frac{(2\nu-1)(2\nu)_j}{j!}B_r(j+1, 2\nu-1) \right]^{l} \beta_{k-l}^{\nu},
\end{equation*}  
and so does (33) where now 
\begin{equation*}
\psi(s) \equiv \psi^{\nu}(s):= \sum_{l=1}^{\infty} (-s)^{l+1} \sum_{j \geq 0} \left[\frac{(2\nu-1)(2\nu)_j}{j!}B_r(j+1, 2\nu-1) \right]^{l}.  
\end{equation*}
Finally, integrating with respect to $s$ leads to 
\begin{align*}
\log(\beta(s)) & = - \sum_{j \geq 0} \sum_{l=1}^{\infty}  \frac{1}{l} \left[-\frac{(2\nu-1)(2\nu)_j}{j!}sB_r(j+1, 2\nu-1) \right]^{l}
\\& = \sum_{j \geq 1}\log\left(1+s\frac{(2\nu-1)(2\nu)_{j-1}}{(j-1)!}B_r(j, 2\nu-1) \right).
\end{align*}
The rest of the proof is similar. 
\end{proof}
From \eqref{BinMom}, we readily obtain: 
\begin{multline*}
\mathbb{E}(N_r(N_r-1)) = \left\{\sum_{j \geq 0}\frac{(2\nu-1)(2\nu)_j}{j!}B_r(j+1, 2\nu-1) \right\}^2- \sum_{j \geq 0} \left[\frac{(2\nu-1)(2\nu)_j}{j!}B_r(j+1, 2\nu-1) \right]^{2} 
\end{multline*}
whence 
\begin{align*}
V(N_r) = \sum_{j \geq 0}\frac{(2\nu-1)(2\nu)_j}{j!}B_r(j+1, 2\nu-1) - \sum_{j \geq 0} \left[\frac{(2\nu-1)(2\nu)_j}{j!}B_r(j+1, 2\nu-1) \right]^{2}.
\end{align*}
If $\nu =1$, then 
\begin{equation*}
\frac{(2\nu-1)(2\nu)_j}{j!}B_r(j+1, 2\nu-1) = r^{2(j+1)},
\end{equation*} 
and one retrieves Peres-Virag's result \eqref{PV}. 

{\bf Acknowledgments}: The authors are grateful to Tomoyuki Shirai for his helpful remarks.

\end{document}